\documentclass[colum]{article}

\usepackage{amsmath}
\usepackage{amsthm}
\usepackage{amsfonts}
\usepackage[top=19truemm,bottom=19truemm,left=20truemm,right=20truemm]{geometry}
\usepackage[pdftex]{graphicx}

\title{Delete Nim}
\author{\large{Koki Suetsugu} \\National Institute of Informatics \\suetsugu.koki@gmail.com \and \large{Tomoaki Abuku} \\University of Tsukuba \\buku3416@gmail.com}
\date{}
\newtheorem{theorem}{Theorem}

\newtheorem{definition}{Definition}
\newtheorem{example}{Example}
\begin{document}
\maketitle

\section{Introduction}

Combinatorial game theory is the mathematical study for strategy of perfect information games in which there are neither chance moves, nor hidden information.
Among early results of combinatorial game theory there is a winning strategy for {\em Nim} by Bouton \cite{Bou01} in 1902. 
Nim is a two-player game with some heaps of stones in which the player to move chooses one of the heaps and takes away arbitrary numbers of stones from it. The upper part of Fig. \ref{nim32} shows a game position of Nim. In the position, there are two heaps of stones: one has three stones and the other has two stones. Since each player can remove any number of stones from a single heap, the candidates for the next positions are shown in the lower part of the figure. The winner of Nim is the player who takes away the last stone. We express Nim positions by the number of stones in each heap like $(3,2)$. 

\begin{figure}[htb]
\centering
\includegraphics[scale = 0.2]{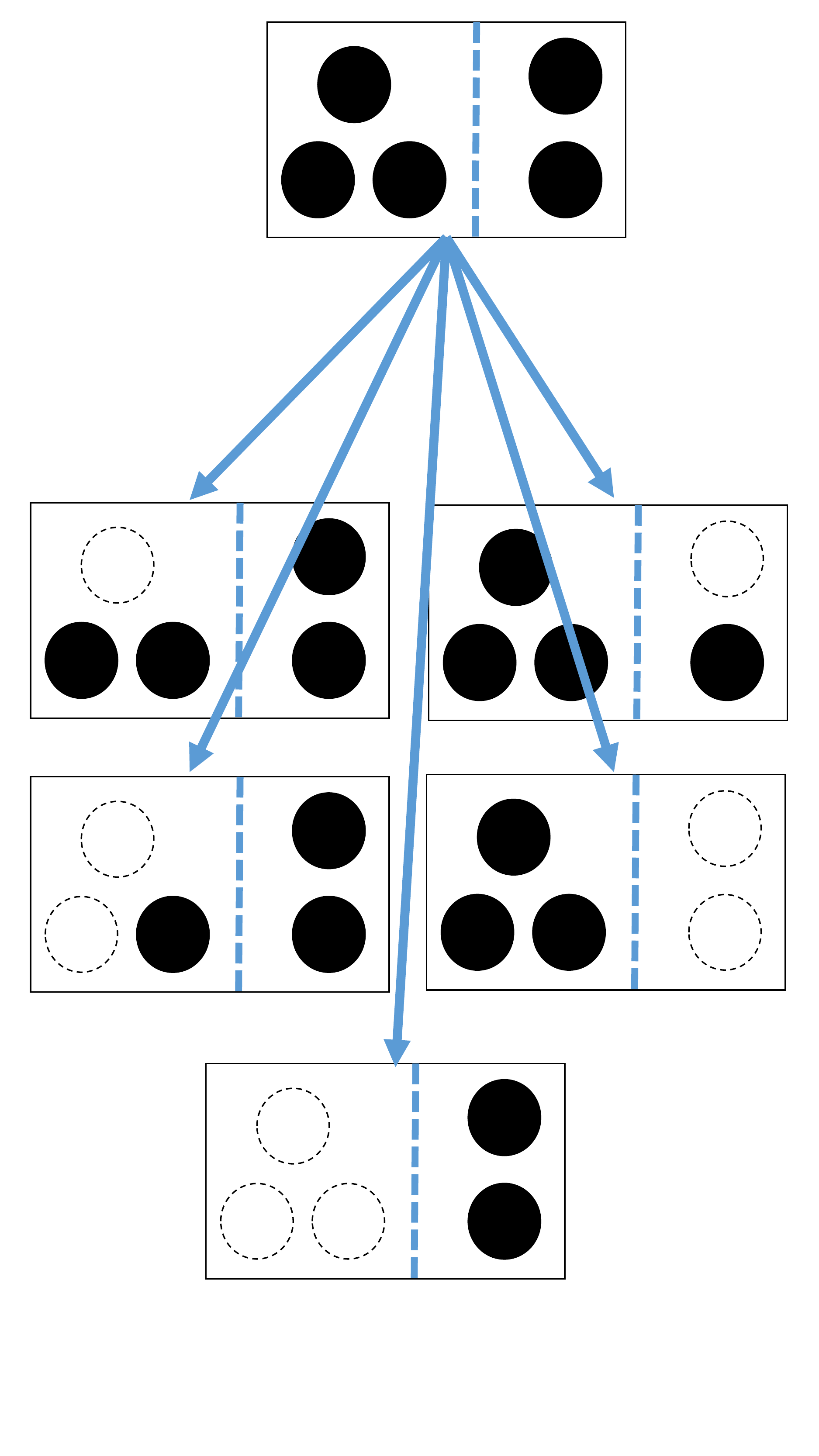}
\caption{The Nim position $(3,2)$}
\label{nim32}

\end{figure}

We say a game is in {\em normal play} if we define that the winner of the game is the player who moves last (like Nim). A game is called {\em impartial} if both players have the same set of options (like Nim). In this paper, we study only impartial games in normal play and assume every play ends in a finite number of moves no matter how the players move.

We say a player has a {\em winning strategy} if she can win regardless of her opponent's move. In an impartial game, we say that a game position is an {\em $\mathcal{N}$-position} or a {\em $\mathcal{P}$-position} if the next player (i.e. the current player) or the previous player (i.e. the other player) has a winning strategy, respectively. 
Clearly, a game position of an impartial two-player game is an $\mathcal{N}$-position or a $\mathcal{P}$-position.



We can analyze whether a Nim position is an $\mathcal{N}$-position or a $\mathcal{P}$-position in a simple way.

\begin{theorem}[Bouton \cite{Bou01}]
\label{Bouton}

A Nim position $(n_{1},n_{2},\ldots,n_{k})$ is a $\mathcal{P}$-position if and only if $n_{1} \oplus n_{2} \oplus \cdots \oplus n_{k} =0$, where operation $\oplus$ (called Nim sum) means the so-called exclusive OR operation in binary expression. 

\end{theorem}

\begin{example}
$2 \oplus 5 \oplus 7 = 10_2 \oplus 101_2 \oplus 111_2 = 0$, therefore, Nim position $(2,5,7)$ is a $\mathcal{P}$-position.
\end{example}

\begin{example}
$4 \oplus 5 \oplus 6 = 100_2 \oplus 101_2 \oplus 110_2 = 111_2 = 7$, therefore, Nim position $(4,5,6)$ is an $\mathcal{N}$-position.
\end{example}

\subsection{Impartial game and $\mathcal{G}$-value}

Sprague \cite{Spr35} and Grundy \cite{Gru39} extended Bouton's theorem for general impartial games in normal play.

\begin{definition}
Let $\mathbb{N}$ be the set of all non-negative integers. For any proper subset $S$ of $\mathbb{N}$, we define minimal excluded function {\em ${\rm mex}(S)$} as follows:
$$
{\rm mex}(S)={\rm min}(\mathbb{N} \setminus S).
$$
\end{definition}
\begin{definition}
For any game position $g$, we define $\mathcal{G}$-value function $\mathcal{G}(g)$ as follows:

$$
\mathcal{G}(g) = {\rm mex}(\{\mathcal{G}(g') \mid g \rightarrow g' \}),
$$
where $g \rightarrow g'$ means that $g'$ is an option of $g$.
\end{definition}

\begin{theorem}[Sprague \cite{Spr35} and Grundy \cite{Gru39}]

For any game position $g$, $g$ is a $\mathcal{P}$-position if and only if $\mathcal{G}(g)=0$.

\end{theorem}

\begin{definition}

For any game positions $g$ and $h$, we define the {\em disjoint sum} $g+h$ recursively as the game whose options are $g+h'$ or $g'+h$ where $g'$ ranges all options of $g$ and $h'$ ranges all options of $h$.

\end{definition}

\begin{theorem}[Sprague \cite{Spr35} and Grundy \cite{Gru39}]

For  any game positions  $g$ and $h$, 

$$
\mathcal{G}(g+h)=\mathcal{G}(g) \oplus \mathcal{G}(h).
$$

\end{theorem}

Therefore, people have been interested in $\mathcal{G}$-values of games and some of the early results show us various structures of $\mathcal{G}$-values in some specific games, for example, Welter's game \cite{Wel54}, cyclic Nimhoff \cite{FL91}, and Lim \cite{FFS14}.







\section{Delete Nim}

Here we define an impartial game called {\em Delete Nim}.
In this game, there are two heaps of stones.
The player chooses one of the heaps and delete the other heap.
Next, she takes away $1$ stone from the chosen heap and optionally splits it into two heaps.

Fig. \ref{deletenim} shows a play of Delete Nim.
\begin{figure}[htb]
\centering
\includegraphics[scale = 0.2]{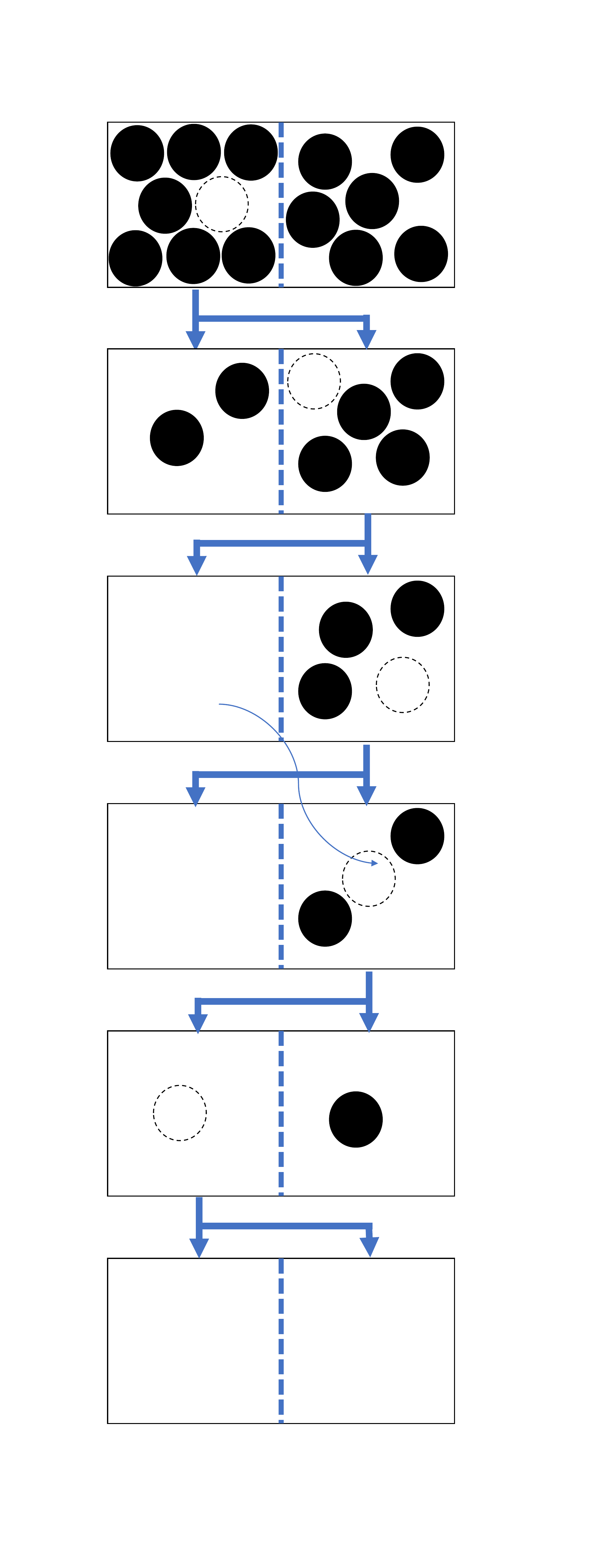}
\caption{one game of Delete Nim}
\label{deletenim}

\end{figure}

In this paper, we show how to compute the $\mathcal{G}$-value of Delete Nim.

To compute the $\mathcal{G}$-value of a position of the game, we need to prepare some definitions and notations.

\begin{definition}
We denote the usual OR  operation of two numbers in binary notations by $\vee$.  
\end{definition}

\begin{example}
$3 \vee 5 = 11_2 \vee 101_2 = 111_2 = 7$.
\end{example}

\begin{example}
$9 \vee 12 = 1001_2 \vee 1100_2 = 1101_2 = 13$. 
\end{example}

Now we can compute the $\mathcal{G}$-value of a position of Delete Nim.

\begin{theorem}
We denote the position of Delete Nim with two heaps of $x$ stones and $y$ stones by $(x,y)$. Then,
\begin{eqnarray}
\mathcal{G}((x,y))&=&v_2((x \vee y)+1) \nonumber,
\end{eqnarray}
where $v_p(n)$ is the p-adic valuation of $n$, that is,
\[v_p(n)=\left \{ \begin{array}{cc}
    {\rm max}\{v \in \mathbb{N} : p^v \mid n\}  & (n\neq 0) \\ 
   \infty & (n=0).
  \end{array}\right.
\]
\end{theorem}

\begin{proof}

Let $x= \sum_{i} 2^{i}x_{i}, y= \sum_{i} 2^{i}y_{i} (x_i,y_i \in \{0, 1\})$ and $h=v_2((x \vee y)+1)$.

First, we show that $(x, y)$ has no next position $(x', y')$ such that $h = v_2((x' \vee y')+1)$ by contradiction. Note that $x' + y' = x-1$ or $x' + y' = y-1$.

If $h=0$, then $x$ and $y$ are even. Therefore, $x' + y'$ is an odd number and $v_2((x' \vee y')+1) \neq 0$, which is a contradiction.

Let $x'= \sum_{i} 2^{i}x'_{i}, y'= \sum_{i} 2^{i}y'_{i} (x'_i,y'_i \in \{0, 1\})$.
If $h > 0$, then $x'_{h} = y'_{h} = 0$ and for any $k < h$, $x'_{k} = 1$ or $y'_{k} = 1$. Therefore, $2^{h}-1\leq ((x'+y') \bmod 2^{h+1})\leq 2^{h+1}-2$ and thus $2^{h}\leq ((x'+y'+1) \bmod 2^{h+1})\leq 2^{h+1}-1$. Then $x_{h} = 1$ or $y_{h} = 1$, which is a contradiction.

Next, we show that for any $h' < h$, $(x, y)$ has a next position such that $h'=v_2((x' \vee y')+1)$.
Since $h=v_2((x \vee y)+1)$, without loss of generality, $x_{h'} = 1$.
Let $x'=x- 2^{h'}$ and $y'=2^{h'}-1$. Clearly, $x'_{h'}=0, y'_{h'}=0, y'_{k} = 1(k<h')$ and $x' + y' = x- 1$. Therefore, $(x', y')$ is a next position of $(x, y)$ and $h'=v_2((x' \vee y')+1)$.

\end{proof}

A game similar to Delete Nim is introduced in \cite{ZT08}. 
The rule is as follows:
There are two heaps of stones. The player chooses one of the heaps and delete it, and she splits the other heap into two heaps.
We call this game a Variant of Delete Nim or VDN.
\begin{theorem}
There is an isomorphism $F(g)$ from the set of all positions of VDN to that of Delete Nim:

$$
F((x,y))=(x-1, y-1).
$$
\end{theorem}

\begin{proof}
Cleary, the end positions of the games hold this isomorphism.
Let $(x' - 1, y' - 1)$ be a next position of $(x - 1, y - 1)$ in Delete Nim. Then without loss of generality, $(x' - 1) + (y' - 1) = (x - 1) - 1$. On the other hand, $(x, y)$ of VDN has the next position $(x', y')$ because $x' + y' = x$. 
For the other side, let $(x', y')$ be a next position of $(x, y)$ in VDN. Then without loss of generality, $x' + y' = x$. On the other hand, $(x-1, y-1)$ of Delete Nim has the next position $(x' - 1, y' - 1)$ because $(x' -1) + (y' - 1)= (x-1) -1$.
Therefore, there is a one-to-one corresponding between $(x - 1, y-1)$ of Delete Nim and $(x, y)$ of VDN.
\end{proof}

In \cite{ZT08}, the $\mathcal{P}$-positions of the game are shown but the $\mathcal{G}$-values of positions of  the game are not discussed. By the isomorphism, now we can compute the $\mathcal{G}$-value of position $(x,y)$ of the game as $v_2((x-1) \vee(y-1)+1)$.

\section{Conclusion}
In this paper, we introduced a game for which one needs to use the OR operation and  $2$-adic  valuation $v_2(n)$ in order to compute the $\mathcal{G}$-value of a position. 
To the best of our knowledge, this is the only case we need the OR operation in order to calculate the $\mathcal{G}$-values.
Therefore, we think that our results would contribute to the field of combinatorial game theory. 

\section{Aknowledgements}
The authors would like to thank Dr. K\^{o} Sakai for valuable discussions and comments. This work was supported by JST CREST Grant Number JPMJCR1401 including AIP challenge program, Japan.
\bibliography{bibsample}
\bibliographystyle{plain}

\end{document}